\documentclass{gtpart}  

\title{Alperin's Fusion Theorem for localities}

\author{R\'emi Molinier}
\givenname{R\'emi}
\surname{Molinier}
\address{Univ. Grenoble Alpes, CNRS, Institut Fourier, F-38000 Grenoble, France}
\email{remi.molinier@univ-grenoble-alpes.fr}
\urladdr{}


\newtheorem{thm}{Theorem}[section]    
\newtheorem{lem}[thm]{Lemma}          
\newtheorem{prop}[thm]{Proposition}
\newtheorem{cor}[thm]{Corollary}

\newtheorem{thm*}{Theorem}
\newtheorem{prop*}{Proposition}

\theoremstyle{definition}
\newtheorem{defi}[thm]{Definition}

\newtheorem*{ex}{Example}

\usepackage{fullpage}
\usepackage{enumerate}
\usepackage[all]{xy}

\numberwithin{equation}{subsection}

\newcommand{\Zp}{\mathbb{Z}_{(p)}}

\newcommand{\Db}{\mathbb{D}}

\newcommand{\Fc}{\mathcal{F}}
\newcommand{\Li}{\mathcal{L}}
\newcommand{\T}{\mathcal{T}}

\newcommand{\A}{\mathcal{A}}

\newcommand{\limproj}[1]{\lim\limits_{\substack{\longleftarrow \\ #1}}}

\newcommand{\Ob}{\text{Ob}}
\newcommand{\Hom}{\text{Hom}}
\newcommand{\Aut}{\text{Aut}}

\newcommand{\Mor}{\text{Mor}}

\newcommand{\Syl}{\text{Syl}}

\begin{document}

\begin{abstract}   
In these notes we give a version of the Alperin-Goldschmidt Fusion Theorem for localities.
\end{abstract}

\maketitle

The notion of a locality was introduced by Chermak \cite{Ch} to study $p$-local structures of finite groups and to prove the existence and uniqueness of linking systems associated to saturated fusion systems. Along the way he proved that there is a one-to-one correspondance between localities with saturated fusion system $\Fc$ and transporter systems associated to $\Fc$. Linking systems and transporter systems were introduced by Broto, Levi and Oliver in \cite{BLO2} and Oliver and Ventura in \cite{OV1} respectively. They introduced these categories to study $p$-completed classifying spaces of finite groups and to develop a theory of classifying spaces for saturated fusion systems. Localities give a more group-like point of view on these objects which allows us to use tools from group theory. This paper gives also an example where the setup of localities helps to prove results on transporter systems.  

We present here a version of the Alperin-Goldschmidt Fusion Theorem for localities which generalizes the Alperin-Goldschmidt Fusion Theorem for finite groups (a nice version can be found in \cite[Theorem 1]{St}). Chermak already gave a version of Alperin's Fusion Theorem for proper localities (i.e. localities which correspond to linking systems) in \cite[Proposition 2.17]{Ch}. Here to be able to work with any locality we have to relax a bit his notion of $\Li$-essential subgroups. 
We use this Theorem to get a generalization of the Alperin Fusion Theorem for transporter systems given by Oliver and Ventura in \cite{OV1}. We give an example at the end how this can be applied to calculate limits of functors over a transporter system and we use it with group cohomology.
\vspace{10pt}

\textbf{Acknowledgement.} The author is eternally grateful to Andy Chermak for so many discussions and teaching him so much about partial groups and localities. The author would also like to thank the referee for his fruitful comments.

\section{Localities and fusion systems}

The notions of partial group and locality are due to Chermak and we refer the reader to \cite[Section 2]{Ch} for the definitions and the basic properties. If $\Li$ is a partial group, we will denote by $\mathbb{D}(\Li)$ its domain and by $\Pi:\mathbb{D}(\Li)\rightarrow \Li$ its partial product. We will also say that $(\Li,\Delta, S)$ is a \emph{locality} when $(\Li,S)$ is a locality via $\Delta$ according to \cite[Definition 2.9]{Ch}.

Let $(\Li,\Delta,S)$ be a locality. For $g\in\Li$ we write
\[S_g=\{s\in S\mid (g^{-1},s,g)\in\Db(\Li) \text{ and } s^g\in S\}.\]
More generally, if $w=(g_1,g_2,\dots,g_r)$ is a word of elements of $\Li$ we write $S_w$ for the set of elements $s$ of $S$ such $(g_1^{-1},s,g_1)\in\Db(\Li)$ and $s^{g_1}\in S$, $(g_2^{-1},s^{g_1},g_2)\in\Db(\Li)$ and $(s^{g_1})^{g_2}\in S$, etc... 

Notice that if $w\in \Db(\Li)$ then $S_w\leq S_{\Pi(w)}$. 
We also recall that $w\in\mathbb{D}(\Li)$ if and only if $S_w\in\Delta$. In particular, $S_g\in\Delta$ for all $g\in\Li$.

\begin{lem}[{\cite[ Lemma 2.7]{Ch}}]\label{normalizer of P in Delta}
Let $(\Li,\Delta,S)$ be a locality. 
\begin{enumerate}[(a)]
\item For every $P\in \Delta$, $N_\Li(P)$ is a subgroup of $\Li$.
\item Let $g\in \Li$ and let $P\in\Delta$ with $P^g\in\Delta$. Then, for all $h\in N_\Li(P)$, $(g^{-1},h,g)\in\Db(\Li)$. Moreover, 
          \[\xymatrix{c_g:N_\Li(P)\ar[r] & N_\Li(P^g)}\]
      is an isomorphism of groups.   
\end{enumerate} 
\end{lem}

A fusion system over a $p$-group $S$ is a way to abstract the action of a finite group $G\geq S$ on the subgroups of $S$ by conjugation. We refer the reader to \cite[Chapter I]{AKO} for the definitions, the basic properties and the usual notations about fusion systems. An important example is given by the fusion system of a locality.

\begin{ex}
Let $(\Li,\Delta,S)$ be a locality. The \emph{fusion system of $\Li$ over $S$} is the fusion system $\Fc_S(\Li)$ generated by $\{c_g:S_g\rightarrow (S_g)^g\mid g\in\Li\}$: for $P,Q\leq S$ and $\varphi\in \Hom(P,Q)$, $\varphi\in\Hom_{\Fc_S(\Li)}(P,Q)$ if $\varphi$ is a composite of restrictions of $c_g:S_g\rightarrow (S_g)^g$ for $g\in \Li$.
\end{ex}

We recall that, if $\Fc$ is a fusion system over a $p$-group $S$, we say that $P\leq S$ is \emph{fully normalized} in  $\Fc$ if $|N_S(P)|\geq|N_S(Q)|$ for all $Q\in P^\Fc$.

\begin{lem}[{\cite[Proposition 2.18(c)]{Ch}}]\label{normalizer of fully normalized}
Let $(\Li,\Delta,S)$ a locality and $P\in\Delta$. 
If $P$ is fully normalized in $\Fc_S(\Li)$, then $N_S(P)\in\Syl_p(N_\Li(P))$.
\end{lem}

\section{Alperin's Fusion Theorem}

\begin{defi}
let $G$ be a finite group.
A subgroup $H$ of $G$ is \emph{strongly $p$-embedded} if $H<G$ and for every $g\in G\smallsetminus H$, $p$ does not divide $|H\cap H^g|$.
\end{defi}

\begin{lem}[{\cite[Proposition A.7.(b)]{AKO}}]\label{strongly pembedded lemma}
Let $G$ be a finite group and $S\in\Syl_p(G)$.
If $G$ does not contains a strongly $p$-embedded subgroup, then 
\[G=\langle x\in G \mid S\cap S^x\neq 1\rangle.\]
\end{lem}

\begin{defi}
Let $(\Li,\Delta,S)$ be a locality.
A subgroup $P\in\Delta$ is \emph{$\Li$-essential} if 
\begin{enumerate}[(i)]
\item $N_S(P)\in\Syl_p(N_\Li(P))$ (or, equivalently $P$ is fully normalized in $\Fc_S(\Li)$), and
\item $N_\Li(P)/P$ contains a strongly $p$-embedded subgroup.
\end{enumerate}
Denote by $\Delta^e \subseteq \Delta$ the subcollection of $\Li$-essential subgroups of $S$.
\end{defi}
Notice also that for a subgroup $P\leq S$ to be  $\Li$-essential, we do not require $P$ to be centric in $\Li$ as it is required in \cite{Ch} Definition 2.4.

\begin{thm}\label{AFT}
Let $(\Li,\Delta,S)$ be a locality.
Then, for every $g\in \Li$, there exists $Q_1,Q_2,\dots,Q_n\in\Delta^e\cup\{S\}$ and $w=(g_1,g_2,\dots,g_n)\in\Db(\Li)$ such that,
\begin{enumerate}[(a)]
\item for every $i\in\{1,2,\dots,n\}$, $g_i\in N_\Li(Q_i)$ and $S_{g_i}=Q_i$ ; and\label{AFTi}
\item $S_w=S_g$ and $g=\Pi(w)$.\label{AFTii}
\end{enumerate}
\end{thm}

\begin{proof}
We will say that $g\in\Li$ admits an \emph{essential decomposition} if there exists $Q_1,Q_2,\dots,Q_n\in\Delta^e\cup\{S\}$ and $w=(g_1,g_2,\dots,g_n)\in\Db(\Li)$ such that \eqref{AFTi} and \eqref{AFTii} are satisfied. Notice that we have the followings.
\begin{enumerate}[(1)]
 \item If $g\in \Li$ admits an essential decomposition then $g^{-1}$ admits an essential decomposition.\label{I}
 \item If $(g_1,g_2,\dots,g_n)\in\Db(\Li)$ with $S_{(g_1,g_2,\dots,g_n)}=S_{g_1g_2\cdots g_n}$ and each $g_i$ admits an essential decomposition, then $g_1g_2\cdots g_n$ admits an essential decomposition.\label{P}
\end{enumerate}

Assume Theorem \ref{AFT} is false and, among all $g\in\Li$ which does not admit an essential decomposition, choose $g$ with $|S_g|$ as large as possible. Set $P=S_g$, $P'=P^g$ and $\Fc=\Fc_S(\Li)$.

If $P=S$, then $g\in N_\Li(S)$ and $g$ admits an essential decomposition with $n=1$, $Q_1=S$ and $w=(g)$.
Thus, we can assume that $P<S$.
Since $P=S_g\in\Delta$, Lemma \ref{normalizer of P in Delta} implies that $N_\Li(P)$ is a subgroup of $\Li$. 
\begin{minipage}[c]{0.68\linewidth}
Choose $Q\in \Delta$ fully normalized and $h,h'\in \Li$ such that $P^h=Q$ and $(P')^{h'}=Q$.
By Lemma \ref{normalizer of fully normalized}, $N_S(Q)\in \Syl_p(N_\Li(Q))$.  
Hence, by Lemma \ref{normalizer of P in Delta} and Sylow's Theorem (applied in $N_\Li(Q)$), we can choose $h$ and $h'$ such that $N_S(P)^h\leq N_S(Q)$ and $N_S(P')^{h'}\leq N_S(Q)$. Then $P<N_S(P)\leq S_h$ and $P'<N_S(P')\leq S_{h'}$ and, by maximality of $|S_g|=|P|=|P'|$, $h$ and $h'$ admit an essential decomposition.  
The word $w=(h^{-1},g,h')$ is in $\Db(\Li)$ via $Q$ and $g':=h^{-1}gh'\in N_\Li(Q)$. Thus $g=hg'h'^{-1}$.
\end{minipage}
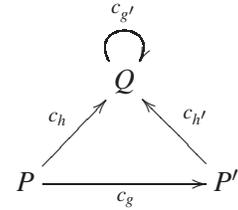
\hfill
\begin{minipage}[c]{0.28\linewidth}
\vspace{-0.5cm}
\[
\xymatrix{ & Q \ar@(ul,ur)[]^{c_{g'}} &\\
          P \ar[ur]^{c_h}\ar[rr]_{c_g} & & P'\ar[lu]_{c_{h'}}
          }
\]
\end{minipage}
Since $g'\in N_\Li(Q)$ and $g=hg'h'^{-1}$,  $P\leq S_{(h,g',h'^{-1})}\leq S_g = P$, i.e. $S_{(h,g',h'^{-1})}= S_g$. Therefore, if $g'$ admits an essential decomposition, then, by \eqref{I} and \eqref{P}, $g$ admits an essential decomposition. Now, if $Q< S_{g'}$, then the maximality of $|S_g|=|P|=|Q|$ implies that $g'$ admits an essential decomposition. Thus $Q=S_g'$ and, up to replace $g$ by $g'$, we can assume that $P$ is fully normalized and $g\in N_\Li(P)$.

If $P\in\Delta^e$, then $g$ admits an essential decomposition with $n=1$, $Q_1=P$ and $w=(g)$.
Thus, $P\in\Delta\smallsetminus \Delta^e$ and $N_\Li(P)/P$ does not contain a strongly $p$-embedded subgroup. Since $P$ is fully normalized, $N_S(P)\in\Syl_p(N_\Li(P))$ and, by Lemma \ref{strongly pembedded lemma}, we can write $g$ as a product $g=g_1g_2\cdots g_n$ with, for $i\in\{1,2,\dots,n\}$, $P<N_S(P)\cap N_S(P)^{g_i}$. We have $P\leq S_{(g_1,g_2,\dots,g_n)}\leq S_g=P$ and so $S_{(g_1,g_2,\dots,g_n)}=S_g$. Now, as $P<N_S(P)\cap N_S(P)^{g_i}\leq S_{g_i}$, the maximality of $|S_g|=|P|$ yields that each $g_i$ admits an essential decomposition. Therefore, by \eqref{P}, $g$ admits an essential decomposition which contradicts the initial assumption.
\end{proof}

As mention before, for a subgroup $P\leq S$ to be $\Li$-essential, we do not require $P$ to be centric. Indeed, Theorem \ref{AFT} does not work if we add this requirement.

\begin{ex}
Let $\Li= \Sigma_3$ be the symmetric group over 3 letters and $p=2$. let $S=<(1,2)>$ and $\Delta$ be the collection of all the subgroups of $S$ (i.e. $\Delta=\{S,\{e\}\}$). Then $(\Li,\Delta,S)$ is a locality where $S$ and $\{e\}$ are $\Li$-essential. But only $S$ is centric and $(1,2,3)$ is not the product of elements in $N_\Li(S)=S$.
\end{ex}

\section{Application to transporter systems}

\subsection{Transporter systems and localities}

We refer the reader to \cite[Definition 3.1]{OV1} for the definition of transporter system.
The typical example is the following.
\begin{ex}
Let $G$ be a finite group $S\in\Syl_p(G)$ and $\Delta$ an $\Fc$-invariant collection of subgroups of $S$.
The \emph{transporter category} of $G$ over $S$ with set of object $\Delta$ is the small category $\T=\T_S^\Delta(G)$ with set of objects $\Delta$ and, for $P,Q\in\Delta$, 
\[\Mor_\T(P,Q)=\{g\in G\mid P^g\leq Q\}.\]
By \cite[Proposition 3.12]{OV1}, it is a transporter system associated to $\Fc_S(G)$. 
\end{ex}

In \cite[Appendix X]{Ch}, Chermak gives a one-to-one correspondence between localities with fusion system $\Fc$ and transporter systems associated to $\Fc$.
One direction of this correspondence is given by the following construction.

\begin{defi}
Let $(\Li,\Delta,S)$ be a locality.
We define the \emph{transporter system} of $(\Li,\Delta,S)$ as the category $\T_{\Delta}(\Li)$ with set of object $\Delta$ and with, for $P,Q\in \Delta$,
\[\Mor_\T(P,Q)= \{g\in \Li\mid P\leq S_g \text{ and } P^g\leq Q\}.\]
By \cite[Lemma X.1]{Ch}, this define a transporter system associated to $\Fc_S(\Li)$. 
\end{defi}

\begin{prop}[{\cite[Proposition X.9]{Ch}}]\label{p:realisable}
Let $\Fc$ be a saturated fusion system over a $p$-group $S$, and $\T$ be an associated transporter system.
Then there exists a locality $(\Li,\Delta,S)$ such that $\T=\T_\Delta(\Li)$.
\end{prop}

\subsection{Alperin's Fusion system for transporter systems}

\begin{defi}
Let $\Fc$ be a saturated fusion system and $\T$ be an associated transporter system.
A subgroup $P\in\Ob(\T)$ is \emph{$\T$-essential} if 
\begin{enumerate}[(i)]
\item $N_S(P)\in\Syl_p(\Aut_\T(P))$, and
\item $\Aut_\T(P)/P$ contains a strongly $p$-embedded subgroup.
\end{enumerate}
Denote by $\T^e$ the full subcategory of $\T$ with set of objects $S$ and all the $\T$-essential subgroups of $S$.
\end{defi}

Let $(\Li,\Delta,S)$ be a locality. By definition, a subgroup $P\leq S$ is $\T_\Delta(\Li)$-essential if and only if $P$ is $\Li$-essential.

The following Theorem gives a generalization of the Alperin Fusion Theorem for transporter systems \cite[Proposition 3.10]{OV1}.

\begin{thm}\label{AFT transp}
Let $\Fc$ be a fusion system over a $p$-group $S$.
If $\T$ is a transporter system associated to $\Fc$, then every morphism in $\T$ is a composite of restrictions of automorphisms of $S$ or $\T$-essential subgroups.
\end{thm}

\begin{proof}
By Proposition \ref{p:realisable}, there is a locality $(\Li,\Delta,S)$ such that $\T=\T_\Delta(\Li)$.
Let $P,Q\in \Delta=\Ob(\T)$ and choose $g\in \Mor_\T(P,Q)\subseteq \Li$. Without loss of generality, we can assume that $P=S_g$ and $Q=S_g^g$.
By Theorem \ref{AFT}, there exists $Q_1,Q_2,\dots,Q_n\in\Delta^e\cup\{S\}$ and $w=(g_1,g_2,\dots,g_n)\in\Db(\Li)$ such that,
\begin{enumerate}[(i)]
\item for every $i\in\{1,2,\dots,n\}$, $g_i\in N_\Li(Q_i)$ and $S_{g_i}=Q_i$; and
\item $S_w=S_g$ and $g=\Pi(w)$.
\end{enumerate}
We have $P=S_g=S_w\leq S_{g_1}=Q_1$ and inductively, for $1\leq i\leq n-1$,  $P^{g_1\cdots g_i}\leq Q_{i+1}=S_{g_{i+1}}$. 
Thus $g$ is the composite of the restriction of $g_i\in \Aut_{\T}(Q_i)$ to $g_i\in \Mor_\T(P^{g_1\cdots g_{i-1}},P^{g_1\cdots g_i})$, for $1\leq i\leq n$.
\end{proof}

The following Corollary, which is a direct consequence of Theorem \ref{AFT transp}, may be helpful when computing limits over transporter systems. 
\begin{cor}\label{C1}
Let $\Fc$ be a fusion system over a $p$-group $S$.
Let $\T$ be a transporter system associated to $\Fc$ and let
\[\xymatrix{F:\T\ar[r] & \A}\]
be a functor into an abelian category $\A$
Then, 
\[\limproj{\T} F=\limproj{\T^e} F.\]
\end{cor}

We can for example use the previous corollary with $\Fc=\Fc_S(G)$, where $G$ is a finite group and $S\in\Syl_p(G)$, and $F=H^*(-,M)$ for $M$ a $\Zp[G]$-module.

\begin{cor}
Let $G$ be a finite group and $S\in \Syl_p(G)$.
Let $M$ a $\Zp[G]$-module.
Then 
\[H^*(G,M)\cong \limproj{\T_S^e(G)} H^*(-,M).\]
\end{cor}

\begin{proof}
By the Cartan-Eilenberg Theorem \cite[Chap XII, Theorem 10.1]{CE},
\[H^*(G,M)\cong \limproj{\T_S(G)} H^*(-,M).\]
Then we can apply Corollary \ref{C1}.
\end{proof}

A more general version was actually proved by Grodal in \cite[Corollary 10.4]{Gr} using deep algebraic topology. The proof we give here is algebraic and more elementary.

\bibliography{biblio}{}
\bibliographystyle{abstract}

\end{document}